\documentclass[12pt,a4paper,reqno]{amsart}

\usepackage[utf8]{inputenc}
\usepackage{amssymb,latexsym}
\usepackage{amsmath,amsthm}
\usepackage{enumerate}
\usepackage[all]{xy} \CompileMatrices
\usepackage{amscd}
\usepackage{slashed}
\usepackage{tikz}
\usepackage{float}
\usetikzlibrary{matrix}
\usepackage{enumitem}
\usepackage{url}
 
\usepackage[a4paper,
left=1in,right=1in,top=1.2in,bottom=1.2in,%
footskip=.25in]{geometry}

\usepackage[colorlinks=false,pdfborderstyle={/S/U/W 0}]{hyperref}

\SelectTips{cm}{12}
\theoremstyle{plain}
\newtheorem{theorem}{Theorem}[section]
\newtheorem{lemma}[theorem]{Lemma}
\newtheorem{corollary}[theorem]{Corollary}
\newtheorem{proposition}[theorem]{Proposition}

\theoremstyle{definition}
\newtheorem{definition}[theorem]{Definition}
\newtheorem{definition-theorem}[theorem]{Definition-Theorem}

\theoremstyle{remark}
\newtheorem{remark}[theorem]{Remark}

\numberwithin{equation}{section}
\setlist[itemize]{leftmargin=*}
 
\newcommand{\tr}{\mathrm{tr}}

\newcommand{\Ker}{\mathrm{Ker}}

\newcommand{\surj}{\to\kern-1.8ex\to}

\newcommand{\Diff}{\mathrm{Diff}}

\newcommand{\Conf}{\mathrm{Conf}}

\newcommand{\Met}{\mathrm{Met}}
\newcommand{\Ric}{\mathrm{Ric}}

\newcommand{\cE}{\mathcal{E}}

\newcommand{\cR}{\mathcal{R}}

\newcommand{\cI}{\mathcal{I}}
\newcommand{\dd}{\mathrm{d}}

\newcommand{\owedge}{\mathbin{\bigcirc\mspace{-15mu}\wedge\mspace{3mu}}}
 
\begin{document}

\title[Rigidity of torsionless three-dimensional Heterotic solitons]{Torsionless three-dimensional Heterotic solitons with harmonic curvature are rigid}

\author[Andrei Moroianu]{Andrei Moroianu}
\address{Université Paris-Saclay, CNRS,  Laboratoire de mathématiques d'Orsay, 91405, Orsay, France, 
and Institute of Mathematics “Simion Stoilow” of the Romanian Academy, 21 Calea Grivitei, 010702 Bucharest, Romania}
\email{andrei.moroianu@math.cnrs.fr}

\author[Miguel Pino Carmona]{Miguel Pino Carmona} 
\author[C. S. Shahbazi]{C. S. Shahbazi} \address{Departamento de Matem\'aticas, Universidad UNED - Madrid, Reino de Espa\~na}
\email{mpino185@alumno.uned.es}
\email{cshahbazi@mat.uned.es} 

\thanks{A.M. was partly supported by the PNRR-III-C9-2023-I8 grant CF 149/31.07.2023 {\em Conformal Aspects of Geometry and Dynamics}. M.P.C. was supported by the UNED-Santander 2024 predoctoral fellowship of the Santander Open Academy foundation. C.S.S. was partially supported by the research grant PID2023-152822NB-I00 of the Ministry of Science of the government of Spain. A.M. and C.S.S. would like to thank the \emph{Mathematisches Forschungsinstitut Oberwolfach} for its hospitality and for providing a stimulating research environment.}

\begin{abstract}
We prove the following rigidity result: every compact three-dimensional Heterotic soliton with vanishing torsion and harmonic curvature is rigid, namely, it is an isolated point in the moduli space. 
\end{abstract}

\maketitle


\section{Introduction}
\label{sec:intro}


The primary objective of this note is to prove a rigidity result for the Heterotic soliton system with auxiliary vanishing torsion on a compact three-manifold. Inspired in Heterotic supergravity \cite{BRI,BRII}, the general Heterotic system was introduced in \cite{MPS26}, where its compact three-dimensional solutions with non-vanishing parallel torsion were classified in terms of hyperbolic three-manifolds and compact quotients of the Heisenberg group equipped with a left-invariant metric. The classification of compact Heterotic solitons with vanishing auxiliary torsion is notably more difficult and remains an open problem. For vanishing auxiliary torsion, the three-dimensional Heterotic soliton system for non-flat solitons reduces to the following system of equations:
\begin{eqnarray}
\Ric^g +  \nabla^g\dd\phi - \frac12{e^{2\phi}}g + \kappa\, \cR^g \circ_g\cR^g &= 0 \,, \label{eq:einstein}	 \\
\dd_{\nabla^g}^{\ast}\cR^g + \cR^g(\dd\phi) &= 0 \,, \label{eq:yangmills} \\ 
\delta^g\dd\phi + \vert\dd\phi\vert^2_g  - e^{2\phi}  +  \kappa\, |\cR^g|^2_g &= 0 \,, \label{eq:dilaton}	
\end{eqnarray} 

\noindent
for couples $(g,\phi)$ consisting of a Riemannian metric $g$ on a three-manifold $M$ and a function $\phi\in C^{\infty}(M)$, the so-called \emph{dilaton} of the system. Here $\nabla^g$ denotes the Levi-Civita connection of $g$, $\Ric^g$ is its Ricci tensor, $\cR^g$ is its Riemann tensor, $\dd_{\nabla^g}^{\ast}$ is the formal adjoint of the exterior covariant derivative $\dd_{\nabla^g}$ associated to $\nabla^g$, $\delta^g$ is the formal adjoint of the exterior derivative $\dd$, and where we have defined:
\begin{gather*}
(\cR^g\circ_g \cR^g)(v_1,v_2) := \langle v_1 \lrcorner \cR^g, v_2\lrcorner \cR^g\rangle_g = \frac{1}{2}\sum_{i,j,k=1}^3 \cR^g_{v_1 , e_i}(e_j,e_k) \cR^g_{v_2 , e_i}(e_j, e_k) \,, \\
\vert \cR^g \vert_g^2 : = \langle \cR^g , \cR^g \rangle_g = \frac{1}{2} \sum_{i,j=1}^3 \langle \cR^g_{e_i,e_j} , \cR^g_{e_i,e_j} \rangle_g = \frac{1}{4}\sum_{i,j,k,l=1}^3 \cR^g_{e_i,e_j}(e_k,e_l)\, \cR^g_{e_i,e_j}(e_k,e_l)    
\end{gather*}

\noindent
in terms of any local orthonormal frame $e_i$ and vectors $v_1,v_2 \in \mathfrak{X}(M)$. The symbol  $\langle \cdot , \cdot  \rangle_g$ denotes the \emph{determinant} metric associated to $g$. Its associated norm is denoted by $\vert \cdot \vert_g^2$. The system of equations \eqref{eq:einstein}, \eqref{eq:yangmills}, and \eqref{eq:dilaton} is the object of study of this note. Equation \eqref{eq:einstein} is the \emph{Einstein equation} of the system, whereas \eqref{eq:yangmills} is a Yang-Mills type of equation for the Levi-Civita connection. On the other hand, Equation \eqref{eq:dilaton} is typically called the \emph{dilaton equation} of the system. Solutions $(g,\phi)$ to equations \eqref{eq:einstein}, \eqref{eq:yangmills}, and \eqref{eq:dilaton} are called \emph{Heterotic solitons with vanishing torsion} or {\em torsionless Heterotic solitons}. \medskip
 
\noindent
Currently, all known three-dimensional torsionless compact Heterotic solitons are Einstein with constant dilaton. If non-flat, such a Heterotic soliton is hyperbolic of scalar curvature $-24 \kappa^{-1}$. We refer to these solitons simply as \emph{hyperbolic}, since the hyperbolicity of the metric forces a compact torsionless Heterotic soliton to have constant dilaton. It remains a major open problem to determine whether compact, torsionless Heterotic solitons with non-constant dilaton actually exist. A natural strategy for constructing a non-flat, torsionless Heterotic soliton with a non-constant dilaton is to deform a hyperbolic one. The main theorem of this note proves that this is not possible, not even if we allow the curvature of $g$ to be harmonic instead of strictly hyperbolic. 

\begin{theorem} 
\label{thm:rigidityessentialharmonic}
Let $(g,\phi)$ be a torsionless, non-flat, three‑dimensional compact Heterotic soliton with vanishing torsion and harmonic curvature. Then, the vector space of essential deformations of $(g,\phi)$ is zero-dimensional.
\end{theorem}
 
\noindent
The notion of \emph{essential deformation}  that we consider is the natural generalization of Koisos's \cite{Koiso}. That is, an essential deformation is an infinitesimal deformation orthogonal to the infinitesimal deformations generated by the action of diffeomorphisms. By the previous theorem, torsionless Heterotic solitons with harmonic curvature are infinitesimally rigid. In turn, due to the existence of a slice and a local Kuranishi model, which follow from the general theory of Diez and Rudolph \cite{DiezRudolph,DiezRudolphII}, this implies that they are also isolated points in the moduli space of torsionless Heterotic solitons. Alternatively, we can think of Theorem \ref{thm:rigidityessential} as a natural generalization of the Mostow rigidity theorem to the Heterotic soliton system. \medskip

\noindent
Theorem \ref{thm:rigidityessentialharmonic} is proved in two steps. First, in Theorem \ref{thm:rigidityessential} we prove that three-dimensional \emph{hyperbolic} compact Heterotic solitons are infinitesimally rigid, which, combined with Corollary \ref{cor:rigidity}, proves rigidity. Then, in Theorem \ref{thm:harmonic} we prove that non-flat three-dimensional compact Heterotic solitons with harmonic curvature are, in fact, hyperbolic, and consequently also rigid. 
 

\section{Preliminaries}


Let $M$ be an oriented three-dimensional manifold equipped with a Riemannian metric $g$. Throughout the paper, when there is no risk of confusion, we will identify 1-forms with vector fields and 2-forms with skew-symmetric endomorphisms using the metric. In our conventions, the Riemann tensor $\cR^g$ of $g$ is defined by:
\begin{equation*}
\cR^g_{v_1,v_2} v_3 = \nabla^g_{v_1} \nabla^g_{v_2} v_3 - \nabla^g_{v_2} \nabla^g_{v_1} v_3 - \nabla^g_{[v_1,v_2]} v_3 \,, \qquad v_1, v_2, v_3\in \mathfrak{X}(M) \,.
\end{equation*}

\noindent
We will understand the curvature tensor $\cR^g$ as a section of $\Omega^2(M) \otimes \Omega^2(M)$, upon identifying skew-symmetric endomorphisms with two-forms, using the underlying Riemannian metric $g$. The Ricci tensor $\Ric^g$ is defined in our conventions as follows:
\begin{equation*}
\Ric^g(v_1,v_2) = \sum_i^3 \cR^g_{e_i v_1}(v_2,e_i) \,, \qquad v_1,v_2 \in \mathfrak{X}(M) \,,
\end{equation*}
where $\{e_1,e_2,e_3\}$ is an orthonormal frame . The scalar curvature is given by $s_g=\tr_g\Ric^g$. \medskip

\noindent
For further reference, we in introduce the Kulkarni-Nomizu product $\owedge$ acting on symmetric tensors $A, B \in \Gamma(T^*M \otimes T^*M)$ as follows:
\begin{align*}
(A \owedge B)(v_1, v_2, v_3, v_4) &= A(v_1,v_3)B(v_2,v_4) + A(v_2,v_4)B(v_1,v_3) \\
&\hspace{1em} - A(v_1,v_4)B(v_2,v_3) - A(v_2,v_3)B(v_1,v_4) \, , 
\end{align*}

\noindent
for every $v_1, v_2, v_3, v_4 \in \mathfrak{X}(M)$. In dimension three, we can write the Riemannian tensor in terms of the Ricci tensor and the scalar curvature:
\begin{equation}\label{eq:riemann1}
\cR^g = - g \owedge \Ric^g + \frac14{s_g}\,g \owedge g \,.
\end{equation}

\noindent
Alternatively, we can write the Riemann tensor as:
\begin{equation}
\label{eq:riemann2}
\cR^g_{v_1,v_2} = \frac12{s_g}\, v_1 \wedge v_2 + v_2 \wedge \Ric^g(v_1) + \Ric^g(v_2)\wedge v_1 \,, \qquad v_1, v_2 \in \mathfrak{X}(M) \,.
\end{equation}

\noindent
Using the previous formula, it is easy to show that the contraction $\cR^g\circ_g \cR^g$ is given by:
\begin{equation}
\label{eq:squareriemann}
\cR^g\circ_g \cR^g =  -  \Ric^g\circ_g \Ric^g +  s_g \Ric^g + \left( \vert \Ric^g\vert^2_g - \frac12{s_g^2} \right) g \,,
\end{equation}

\noindent
where we have defined:
\begin{equation*}
\Ric^g\circ_g \Ric^g(v_1,v_2) := g(\Ric^g(v_1),\Ric^g(v_2)) \,, \qquad v_1, v_2 \in \mathfrak{X}(M) \, .
\end{equation*}

\noindent
In particular, the norm of $\cR^g$ is given by:
\begin{equation}\label{eq:normriemann}
\vert \cR^g \vert_g^2 = \frac{1}{2} \tr_g(\cR^g\circ_g \cR^g)  =  \vert \Ric^g \vert_g^2 - \frac{1}{4} s_g^2 \,.
\end{equation}

\noindent
We also recall the standard contracted second Bianchi identity
\begin{equation}\label{eq:bianchi}
\dd_{\nabla^g}^{\ast}\cR^g(v_1,v_2,v_3) = d_{\nabla^g}\Ric^g(v_2,v_3,v_1) \,, \qquad v_1,v_2,v_3 \in \mathfrak{X}(M) \,.
\end{equation}

\noindent
Using these formulas, the three-dimensional Heterotic soliton system given in Equations \eqref{eq:einstein}, \eqref{eq:yangmills}, and \eqref{eq:dilaton} can be further expanded and simplified.

\begin{proposition}\label{prop:heterotic2}
A pair $(g,\phi)$ is a non-flat Heterotic soliton with vanishing torsion on a three-dimensional manifold $M$ if and only if:
\begin{align}
- \kappa \Ric^g\circ_g\Ric^g + (1 + \kappa s_g) \Ric^g + \frac{1}{3}\left( - s_g - \frac{3\kappa}{4} s_g^2 - |\dd\phi|_g^2 + e^{2\phi} \right) \,g + \nabla^g \dd\phi &= 0 \,, \label{eq:einstein2} \\
\dd_{\nabla^g}^{\ast}\left( - g \owedge \Ric^g + \frac14{s_g}\, g \owedge g \right) + \left( - g \owedge \Ric^g + \frac14{s_g}\, g \owedge g \right)(\dd\phi) &= 0 \,, \label{eq:yangmills2} \\
s_g - 3\delta^g\dd\phi - 2|\dd\phi|_g^2 + \frac12e^{2\phi} &= 0 \,. \label{eq:dilaton2}
\end{align}
\end{proposition}

\begin{proof}
First, subtracting twice the dilaton equation \eqref{eq:dilaton} from the trace of the Einstein equation \eqref{eq:einstein} yields \eqref{eq:dilaton2}. \medskip

\noindent
On the other hand, by the general formula \eqref{eq:riemann1}, the Yang-Mills equation \eqref{eq:yangmills} is equivalent to \eqref{eq:yangmills2}. \medskip

\noindent
Now, using \eqref{eq:squareriemann}, the Einstein equation \eqref{eq:einstein} can be equivalently written
\begin{equation}\label{eq:einsteinexp1}
- \kappa \Ric^g\circ_g\Ric^g + (1 + \kappa s_g) \Ric^g + \left( \kappa \,|\Ric^g|_g^2 - \frac{\kappa}{2} s_g^2 - \frac{1}{2} e^{2\phi} \right) \,g + \nabla^g \dd\phi = 0 \,.
\end{equation}
Taking the trace of \eqref{eq:einsteinexp1} shows
\begin{align*}
\kappa \,|\Ric^g|_g^2 &= \frac{1}{2}\left( - s_g + \frac{\kappa}{2}s_g^2 + \frac{3}{2}e^{2\phi} + \delta^g\dd\phi \right) = \frac{1}{2}\left( - \frac{2}{3}s_g + \frac{\kappa}{2}s_g^2 + \frac{5}{3}e^{2\phi} - \frac{2}{3}|\dd\phi|_g^2 \right) \,,
\end{align*}
where in the second equality we have used \eqref{eq:dilaton2} in order to express $\delta^g\dd\phi$. Substituting back into \eqref{eq:einsteinexp1} yields
\begin{equation*}
- \kappa \Ric^g\circ_g\Ric^g + (1 + \kappa s_g) \Ric^g + \frac{1}{3}\left( - s_g - \frac{3\kappa}{4} s_g^2 - |\dd\phi|_g^2 + e^{2\phi} \right) \,g + \nabla^g \dd\phi = 0 \,,
\end{equation*}
which is \eqref{eq:einstein2}. Doing the computations backwards proves the converse.
\end{proof}

\noindent
In the following, we will consider the two formulations of the three-dimensional Heterotic soliton presented (and, occasionally, a combination of them) as most convenient.

\begin{corollary}\label{cor:ricdphi}
Let $(g,\phi)$ be a three-dimensional Heterotic soliton with vanishing torsion. Then:
\begin{equation}\label{eq:ricdphi}
\Ric^g(\dd\phi) = \frac{1}{2}\dd s_g .
\end{equation}
\end{corollary}

\begin{proof}
Let $\{e_1, e_2, e_3\}$ be a local orthonormal frame. The Yang--Mills equation reads
\begin{equation*}
- \nabla^g_{e_j}\cR^g(e_j,v_1, v_2, v_3) + \cR^g(\dd\phi,v_1, v_2, v_3) = 0, \qquad v_1, v_2, v_3 \in \mathfrak{X}(M) \,,
\end{equation*}
with the summation convention over repeating subscripts. Setting $v_1=v_2=e_i$ and summing over $i$ gives
\begin{equation*}
- \nabla^g_{e_j} \left(\cR^g(e_j,e_i,e_i,v_3)\right) + \cR^g(\dd\phi,e_i,e_i,v_3) = 0 \,.
\end{equation*}
Since $\cR^g(e_i,\cdot,\cdot,e_i)=\Ric^g$ by definition, using the symmetries of $\cR^g$ we obtain
\begin{equation*}
- \nabla^g_{e_j}\Ric^g(e_j,v_3) + \Ric^g(\dd\phi,v_3) = 0 \,.
\end{equation*}
The contracted second Bianchi identity reads $\nabla^g_{e_j}\Ric^g(e_j,v_3)=\frac{1}{2}\,\dd s_g(v_3)$,
which together with the previous equation gives \eqref{eq:ricdphi}.
\end{proof}

\begin{remark}\label{rem:yangmills}
Viewing the Riemann tensor $\cR^g$ as a section of $\Omega^2(M) \otimes \Omega^2(M)$, the Yang--Mills equation becomes an equation in $\Omega(M) \otimes \Omega^2(M)$. Via the Hodge star operator, we obtain an isomorphism $\Omega^1(M) \otimes \Omega^2(M) \cong \Omega^1(M) \otimes \Omega^1(M)$, which decomposes into trace, skew-symmetric, and traceless symmetric parts. Identity \eqref{eq:ricdphi} is precisely the trace component of this decomposition.
\end{remark}

\noindent
We proceed to characterize traceless Heterotic solitons with constant dilaton.
\begin{proposition}\label{prop:classificationphiconstant}
A non-flat compact Heterotic soliton with vanishing torsion $(g,\phi)$ has constant dilaton if and only if it is hyperbolic, in which case $\kappa s_g = - 24$, and $\kappa e^{2\phi}= 48$.
\end{proposition}

\begin{proof}
First, assume that $\phi$ is constant. Then the system \eqref{eq:einstein2}--\eqref{eq:dilaton2} reduces to:
\begin{align}
- \kappa \Ric^g \circ_g \Ric^g + (1 + \kappa s_g) \Ric^g - \Bigl( \frac{\kappa}{4} s_g^2 + s_g \Bigr) g &= 0, \label{eq:einsteinconst} \\
d_{\nabla^g}\Ric^g &= 0, \label{eq:yangmillsconst} \\
s_g + \frac{1}{2} e^{2\phi} &= 0\,, \label{eq:dilatonconst}
\end{align}
where in the first equation we used \eqref{eq:dilatonconst} in order to eliminate $e^{2\phi}$ and in the second equation we used \eqref{eq:bianchi}.
Equation \eqref{eq:dilatonconst} shows that $s_g$ is a negative constant, and Equation \eqref{eq:yangmillsconst} says that $\Ric^g$ is a Codazzi tensor, which has constant eigenvalues thanks to \eqref{eq:einsteinconst}. By \cite[\S16.12]{Besse}, $(M,g)$ is either Einstein or locally a product $\mathbb{R}\times\Sigma$ where $\Sigma$ is a surface of constant curvature. In either case, the eigenvalues are globally constant, and, by Equation \eqref{eq:einsteinconst}, each eigenvalue $\lambda$ satisfies the quadratic equation
\begin{equation}\label{eq:quad}
\kappa\lambda^2 - (1+\kappa s_g)\lambda + \left(\frac{\kappa}{4}s_g^2 + s_g\right) = 0 \,.
\end{equation}

\noindent
We now examine the two possible geometric structures. \medskip

\noindent
If $(M,g)$ is locally a product $\mathbb{R}\times\Sigma$, then $\Ric^g$ has eigenvalues $(0,\mu,\mu)$ with $\mu\neq0$ and $s_g = 2\mu$. Setting $\lambda=0$ in \eqref{eq:quad} gives $\frac{\kappa}{4}s_g^2 + s_g = 0$, whence $s_g = -\frac4\kappa$ (since $s_g\neq0$). Thus $\mu = \frac12{s_g} = -\frac2\kappa$.  
Now the eigenvalue $\mu$ must also satisfy \eqref{eq:quad}; but substituting $\mu=-\frac2\kappa$ and $s_g=-\frac4\kappa$ into the left-hand side of \eqref{eq:quad} gives
\begin{equation*}
\kappa\cdot\frac{4}{\kappa^2} - \left(1-\kappa\cdot\frac{4}{\kappa}\right)\left(-\frac{2}{\kappa}\right) + \left(\frac{\kappa}{4}\cdot\frac{16}{\kappa^2} - \frac{4}{\kappa}\right) = \frac{4}{\kappa} - \frac{6}{\kappa} + \frac{4}{\kappa} - \frac{4}{\kappa} = -\frac{2}{\kappa} \neq 0 \,,
\end{equation*}
a contradiction. Hence, this case cannot occur. \medskip

\noindent
If $(M,g)$ is Einstein, then $\Ric^g = \frac13{s_g}g$.  Substituting $\lambda = \frac13{s_g}$ into \eqref{eq:quad} yields
\begin{equation*}
\frac13{s_g}\left(\frac{\kappa}{12}s_g + 2\right) = 0 \,.
\end{equation*}
Since $s_g\neq0$ (otherwise the soliton would be flat), we obtain $s_g = -\frac{24}\kappa$. Consequently $e^{2\phi} = -2s_g = \frac{48}\kappa$, and $\Ric^g = -\frac8\kappa\,g$, i.e., $(g, \phi)$ is hyperbolic.
\end{proof}


\section{Linearization of certain curvature operators}
\label{sec:curvatureperators}


In this section, we study the linearization of certain curvature operators that appear in the three-dimensional Heterotic soliton system, as a preliminary step toward the linearization of the Heterotic soliton system in Section \ref{sec:essential-deformations}. Let $h \in \Gamma(T^*M \otimes T^*M)$. For a given Riemannian metric $g$, the differential of a given curvature operator along $h$ at $g$ will be denoted by $\dd_g( - ) (h)$. Since $\owedge$ is linear and symmetric, the variation of the Riemann tensor is simply:
\begin{equation}
\label{eq:linearriemann}
\dd_g\cR^g(h) = - h \owedge \Ric^g - g \owedge \dd_g\Ric^g(h) + \frac{1}{4} \dd_gs_g(h) \, g \owedge g + \frac12{s_g} \,g \owedge h . 
\end{equation}
Standard computations give the following identities (see, e.g., \cite[\S1.174]{Besse} or \cite{Bunk}):
\begin{align}
\dd_g\Ric^g(h) &= \frac{1}{2} \Delta^g_L h - \nabla^{S, g} \nabla^{g \ast} h - \frac{1}{2} \nabla^g \dd \, \tr_g(h) \,, \label{eq:linearricci} \\
\dd_gs_g(h) &= \delta^g \dd (\tr_g h) + \delta^g \nabla^{g \ast} h - g(h, \Ric^g) \,, \label{eq:linearscalar}
\end{align}
where $\nabla^{g \ast}$ is the formal adjoint of $\nabla^g$, $\nabla^{S, g}$ is the symmetrization of $\nabla^g$ and 
\begin{equation*}
\Delta^g_L h = \nabla^{g \ast}\nabla^gh + h\circ_g \Ric^g + \Ric^g\circ_g h -2\cR_0(h)
\end{equation*}
is the Lichnerowicz Laplacian, with $\cR_0$ defined through
\begin{equation*}
\cR_0(h)(v_1,v_2) = g(\cR(\cdot,v_1,v_2,\cdot), h) \,.
\end{equation*}

\noindent
In the following, assume that the background metric $g$ is Einstein. These assumptions give the following standard relations:
\begin{equation}
\label{eq:riemannriccieinstein}
\cR^g = -\frac1{12}{s_g} \, g \owedge g \,, \qquad \Ric^g = \frac13{s_g} \,g \, .
\end{equation}

\begin{lemma}\label{lemma:linearcurvature}
Let $(M,g)$ be a three-dimensional Einstein manifold. Then, the linearizations of $\cR^g \circ_g \cR^g$ and $|\mathcal{R}|_g^2$ in the direction of $h \in \Gamma(T^*M \otimes T^*M)$ are given by:
\begin{align}
\dd_g(\cR^g \circ_g \cR^g)(h) &= \frac13{s_g}\, \dd_g\Ric^g(h) - \frac1{18}{s_g^2}\, h \,, \label{eq:linearriemannsquare} \\
\dd_g(|\mathcal{R}|_g^2)(h) &= \frac16{s_g}\,\dd_gs_g(h) \,, \label{eq:linearriemannnorm}
\end{align}
where $\dd_g\Ric^g(h)$ and $\dd_gs_g(h)$ are the linearizations of the Ricci tensor and the scalar curvature.
\end{lemma}

\begin{proof}
Set $3 \lambda = s_g $. From \eqref{eq:squareriemann}, we compute:
\begin{align}
\dd_g(\cR^g \circ_g \cR^g)(h) =& - \dd_g(\Ric^g \circ_g \Ric^g)(h) + \dd_g(s_g \Ric^g)(h) \nonumber\\
& + \dd_g  |\Ric^g|_g^2(h) g -  s_g  \dd_g s_g (h) g  +  ( |\Ric^g|_g^2 - \frac12{s_g^2} ) h \label{eq:lin}   \, .
\end{align}
We evaluate term-wise on the Einstein background, and we consider the linearization of $\Ric^g\in \Gamma(T^{\ast}M\otimes T^{\ast}M)$ as a symmetric two-form on $M$, rather than as an endomorphism of $TM$. Hence, for the first term we have:
\begin{equation*}
\dd_g(\Ric^g \circ_g \Ric^g)(h) = - \lambda^2 h  + 2\lambda\,\dd_g \Ric^g(h) \, .
\end{equation*}
where we have taken into account that the \emph{composition} $(-)\circ_g(-)$ of symmetric two-forms involves the (inverse) metric and must therefore be linearized, yielding the first term on the right-hand side of the previous equation. More precisely, we can write:
\begin{equation*}
(\Ric^g \circ_g \Ric^g)(v_1,v_2) = g^{\ast}(\Ric^g(v_1), \Ric^g(v_2))\, , \qquad v_1, v_2 \in \mathfrak{X}(M) \,,
\end{equation*}

\noindent
where $g^{\ast}$ denotes the metric induced by $g$ on $T^{\ast}M$. Linearizing the \emph{inverse} metric $g^{\ast}$ and evaluating on the given Einstein metric, yields the term $-\lambda^2 h$ in the equation above. For the second term, we have:
\begin{equation*}
\dd_g(s_g \Ric^g)(h) = \lambda \,\dd_gs_g(h) \,g + 3\lambda\,\dd_g\Ric^g(h) \,.
\end{equation*}
For the third term, we compute:
\begin{align*}
 \dd_g|\Ric^g|_g^2(h) &= 2 g^{\ast}(\Ric^g, \dd_g\Ric^g(h)) - 2 g^{\ast}(h, \Ric^g\circ_g \Ric^g)\\& = 2\lambda \,\tr_g[\dd_g\Ric^g(h)] - 2\lambda^2 \,\tr_gh\,,   
\end{align*}

\noindent
where we have taken into account that the norm square of $\Ric^g$ involves twice the induced metric $g^\ast$. Using this equation, the second line in Equation \eqref{eq:lin} simplifies as follows: 
\begin{align*}
 & \dd_g  |\Ric^g|_g^2(h)g -  s_g  \dd_g s_g (h) g  +  ( |\Ric^g|_g^2 - \frac12{s_g^2}) h = \\
 &= \left( 2\lambda \,\tr_g[\dd_g\Ric^g(h)] - 2\lambda^2 \,\tr_g h - 3\lambda \,\dd_gs_g(h) \right) g - \frac{3}2\lambda^2h \, .
\end{align*}

\noindent
Altogether, we obtain:
\begin{align*}
\dd_g(\cR^g \circ_g \cR^g)(h)  =    \lambda\, \dd_g\Ric^g(h) - \frac12{\lambda^2} h  + \left[2\lambda\tr_g(\dd_g\Ric^g(h)) - 2\lambda^2\tr_g h - 2\lambda \,\dd_gs_g(h) \right] \,g   \,,
\end{align*}
but, since the linearized scalar curvature satisfies
\begin{equation*}
\dd_gs_g(h) = - g(h, \Ric^g) + \tr_g(\dd_g\Ric^g(h)) = - \lambda \,\tr_gh + \tr_g(\dd_g\Ric^g(h)) \,,
\end{equation*}
the expression in brackets above cancels identically. Thus:
\begin{equation*}
\dd_g(\cR^g \circ_g \cR^g)(h) = \lambda \dd_g\Ric^g(h) - \frac12{\lambda^2} h \, .
\end{equation*}
Substituting back $\lambda$, \eqref{eq:linearriemannsquare} follows. Taking the trace of \eqref{eq:linearriemannsquare} yields
\begin{align*}
\tr_g [\dd_g(\cR^g \circ_g \cR^g)(h)] &= \lambda \tr_g[\dd_g\Ric^g(h)] - \frac12{\lambda^2}\tr_g h \\
&= \lambda \,\dd_gs_g(h) + \lambda g(h, \Ric^g(h)) - \frac12{\lambda^2} \tr_g h = \lambda \,\dd_gs_g(h) + \frac12{\lambda^2} \tr_g h \,.
\end{align*}
On the other hand, linearizing the relation $\dd_g|\cR^g|_g^2(h) = \frac{1}{2} \dd_g\tr_g (\cR^g \circ_g \cR^g)$ gives
\begin{align*}
\dd_g|\cR^g|_g^2(h) &= - \frac{1}{2} g(h, \cR^g \circ_g \cR^g) + \frac{1}{2} \tr_g [\dd_g(\cR^g \circ_g \cR^g)(h)] \\
&= - \frac14{\lambda^2} \tr_g(h) + \frac{1}{2} \tr_g [\dd_g(\cR^g \circ_g \cR^g)(h)] \,,
\end{align*}
where we have used $\cR^g \circ_g \cR^g = \frac12{\lambda^2} g$ by \eqref{eq:squareriemann}.
Combining both expression produces
\begin{equation*}
\dd_g|\cR^g|_g^2(h) = \frac12{\lambda} \,\dd_gs_g(h) \,.
\end{equation*}
Substituting back $\lambda$, \eqref{eq:linearriemannnorm} follows, completing the proof.
\end{proof}

\begin{remark}
A useful sanity check is linearizing $|\cR^g|_g^2$ directly from its alternative expression \eqref{eq:normriemann}:
\begin{align*}
\dd_g|\cR^g|_g^2(h) &= \dd_g \left( |\Ric^g|_g^2 - \frac14{s_g^2} \right)\\& = - 2 g^{\ast}(h, \Ric^g \circ_g \Ric^g) + 2g^{\ast}(\Ric^g, \dd_g\Ric^g(h)) - \frac{1}{2}s_g \,\dd_gs_g(h) \\
&= - \frac{2}{9}s_g^2 \,\tr_gh + \frac{2}{3}s_g \,\tr_g[\dd_g\Ric^g(h)] - \frac{1}{2}s_g \,\dd_gs_g(h) = \frac16{s_g} \dd_gs_g(h) \,,
\end{align*}
where we have used $\dd_g s_g(h) = - g(h, \Ric^g) + \tr_g[\dd_g\Ric^g(h)]$ and \eqref{eq:riemannriccieinstein}. This is precisely Equation \eqref{eq:linearriemannnorm}.
\end{remark}


\section{Rigidity of hyperbolic Heterotic solitons}
\label{sec:rigidity}


In this section, we study the local structure of the moduli of the torsionless Heterotic soliton system around a hyperbolic soliton, showing that they are isolated points in moduli space and therefore rigid. 


\subsection{Existence of a slice}


Let $\Conf(M)$ denote the configuration space of the torsionless Heterotic system, that is:
\begin{equation*}
\Conf(M) = \Met(M) \times C^{\infty}(M) \,.
\end{equation*}

\noindent
which we consider as a tame Fréchet manifold \cite{Hamilton}. The diffeomorphism group $\Diff(M)$ acts smoothly on the configuration space by pullback: 
\begin{equation*}
\Psi\colon \Conf(M)\times \Diff(M) \to \Conf(M)\, , \qquad ((g,\phi),u) \mapsto (u^*g, \phi\circ u) \,.
\end{equation*}

\noindent
For a fixed soliton $(g,\phi)$ we introduce the orbit map:
\begin{equation*}
\Psi_{g,\phi}\colon \Diff(M) \to \Conf(M), \qquad u \mapsto (u^*g, \phi\circ u) \,.
\end{equation*}
Its differential at the identity $e\in\Diff(M)$ can be shown to be:
\begin{equation*}
\dd_e\Psi_{g,\phi}\colon \mathfrak{X}(M) \to T_{g,\phi}\Conf(M), \qquad
v \mapsto (\mathcal{L}_v g,\, \dd\phi(v)) \,.
\end{equation*}
Thus, its formal $L^2$‑adjoint becomes:
\begin{equation*}
\dd_e\Psi_{g,\phi}^{*}\colon T_{(g,\phi)}\Conf(M) \to \mathfrak{X}(M), \qquad
(h,\xi) \mapsto  2\nabla^{g \ast} h  + \xi \dd\phi \,.
\end{equation*}  
The symbol of $\dd_e\Psi_{g,\phi}$ is injective, and hence it is overdetermined elliptic. Then, from the standard elliptic theory on closed manifolds (see for example \cite{Moroianu:2023jof}), the following $L^2$‑orthogonal decomposition follows:
\begin{equation*}
T_{g,\phi}\Conf(M) = \mathrm{Im}(\dd_e\Psi_{g,\phi}) \oplus \Ker(\dd_e\Psi_{g,\phi}^{*}),
\end{equation*}

\noindent
where both factors are closed subspaces. \medskip

\noindent 
As a direct consequence of the general theory developed by Diez and Rudolph \cite{DiezRudolph}, this action admits a slice \cite[Definition 2.2]{DiezRudolph}. To prove this, we first observe that, by a celebrated result of  \cite{Ebin}, the action of the diffeomorphism group on the Riemannian manifolds of a compact manifold admits a slice $S_g \subset \Met(M)$ through every metric $g\in \Met(M)$. The stabilizer of this action is the isometry group $\mathrm{Iso}(M,g)$ of $g$, which is well-known to be compact by the compactness of $M$. Hence, by \cite[Theorem 3.15]{DiezRudolph} the action of $\mathrm{Iso}(M,g)$ on $C^{\infty}(M)$ admits a smooth slice, which used in combination with \cite[Proposition 3.29]{DiezRudolph} proves the existence of a smooth slice for the action of $\Diff(M)$ on $\Conf(M)$ in the tame Fréchet category. In particular, we obtain the following result.

\begin{proposition}
Let $(g,\phi)\in \Conf(M)$. There exists a smooth submanifold $S_{g,\phi}\subset \Conf(M)$ and an open neighbourhood $U_{g,\phi} \subset \Conf(M)$  of $(g,\phi)$ homeomorphic to $S_{g,\phi}/\cI_{g,\phi}$, where $\cI_{g,\phi}$ denotes the stabilizer of $(g,\phi)$ in $\Diff(M)$. If $\cI_{g,\phi}$ is trivial, then $U_{g,\phi}$ is a smooth tame Fréchet manifold modeled on the tame Fréchet vector space $\Ker(\dd_e\Psi_{g,\phi}^{*}) \subset T_{g,\phi} \Conf(M)$.
\end{proposition}
 
\noindent
Using the equations of the three-dimensional Heterotic system, we introduce the following smooth maps of Fréchet manifolds:
\begin{align*}
\cE_{\mathrm{E}}: \Conf(M) &\to \Gamma(T^*M \otimes T^*M) \,,& \qquad (g,\phi) &\mapsto \Ric^g + \nabla^g\dd\phi - \frac12{e^{2\phi}}g + \kappa \cR^g \circ_g \cR^g \,, \\
\cE_{\mathrm{YM}}: \Conf(M) &\to \Gamma(T^*M \otimes \Omega^2(M)) \,,& \qquad (g,\phi) &\mapsto \dd_{\nabla^g}^{\ast}\cR^g + \cR^g(\dd\phi) \,, \\
\cE_{\mathrm{D}}: \Conf(M) &\to C^\infty(M) \,,& \qquad (g,\phi) &\mapsto \delta^g\dd\phi + |\dd\phi|_g^2 - e^{2\phi} + \kappa|\cR^g|_g^2 \,.
\end{align*}

\noindent
We set:
\begin{equation*}
\cE = (\cE_{\mathrm{E}}, \cE_{\mathrm{YM}}, \cE_{\mathrm{D}}) \colon \Conf(M) \to \Gamma(T^{*}M\otimes T^{*}M) \times \Gamma(T^*M \otimes \Omega^2(M)) \times C^{\infty}(M) \,.
\end{equation*}
The map $\cE$ is diffeomorphism‑equivariant, that is:
\begin{equation*}
\cE(u^*g, \phi\circ u) = \left(u^*\cE_{\mathrm{E}}(g,\phi),\, u^*\cE_{\mathrm{YM}}(g,\phi) ,\, \cE_{\mathrm{D}}(g,\phi)\circ u\right) \,.
\end{equation*}
Consequently, the moduli space of three‑dimensional Heterotic solitons with vanishing torsion is defined as:
\begin{equation*}
\mathfrak{M}(M) := \frac{\cE^{-1}(0)}{\Diff(M)} \,,
\end{equation*}
equipped with the quotient topology induced by the Fréchet topology on $\Conf(M)$. Due to the existence of a slice $S_{g,\phi}$ through every $(g,\phi) \in \Conf(M)$, if $(g,\phi) \in \cE^{-1}(0)$ then $[g,\phi]\in \mathfrak{M}(M)$ has an open neighbourhood $U_{[g,\phi]}$ homeomorphic to $S_{g,\phi}\cap \cE^{-1}(0)$:
\begin{equation*}
U_{[g,\phi]} \simeq S_{g,\phi}\cap \cE^{-1}(0) \,,    
\end{equation*}

\noindent
after possibly shrinking $S_{g,\phi}$. The \emph{candidate} tangent space of $S_{g,\phi}\cap \cE^{-1}(0)$ is precisely what we call, following Koiso \cite{Koiso}, the vector space of \emph{essential deformations} $\mathbb{E}_{(g,\phi)}$ of $(g,\phi)$, that is:
\begin{equation*}
\mathbb{E}_{(g,\phi)} := \Ker(\dd_{(g,\phi)}\cE) \cap \Ker(\dd_e\Psi_{g,\phi}^{*}) .
\end{equation*}

\noindent
This, together with the existence of a local Kuranishi model by \cite[Theorem 5.3]{DiezRudolph}, provides us with the rigidity criteria that we will apply in the next subsection.

\begin{corollary}
\label{cor:rigidity}
Let $(g,\phi)$ be such that $\mathbb{E}_{(g,\phi)} = 0$. Then, there exists an open neighbourhood of $[g,\phi] \in \mathfrak{M}(M)$ that is homeomorphic to a point. In particular, $[g,\phi]$ is an isolated point in moduli space. 
\end{corollary}


\subsection{Essential deformations}
\label{sec:essential-deformations}


Let $(g,\phi)$ be a three‑dimensional torsionless Heterotic soliton with constant dilaton. By Proposition~\ref{prop:classificationphiconstant}, it is hyperbolic Einstein:
\begin{equation}
\label{eq:backg}
\Ric^{g} = \frac13s_g g \,, \qquad \kappa s_g = - 24 \,, \qquad \kappa e^{2\phi} = 48 \,.
\end{equation}
The existence of a slice implies that the tangent space to the moduli space at $[g,\phi]$ is contained in the intersection of the kernel of the linearized equations with the slice. \medskip
 
\noindent
Let $(h,\xi) \in \mathbb{E}_{(g,\phi)}$. We proceed by linearizing the scalar identity: \begin{equation}\label{eq:scalaridentity}
s_g + |\dd\phi|_g^2 - \frac{5}{2}e^{2\phi} + 3\kappa|\cR^g|_g^2 = 0 \,,
\end{equation}
which follows from adding the dilaton equation \eqref{eq:dilaton} to the trace of the Einstein equation \eqref{eq:einstein}.
Using that $\dd\phi=0$ on the given hyperbolic background $(g,\phi)$ together with Equation \eqref{eq:linearriemannnorm} yields:
\begin{equation}\label{eq:linearscalaridentity}
0 = \dd_gs_g(h) - 5e^{2\phi}\xi + 3\kappa\,\dd_g (\vert \cR^g\vert_g^2)(h) = (1+\frac{\kappa}{2}s_g) \dd_gs_g(h) - 5e^{2\phi}\xi\, .
\end{equation}
Taking the exterior derivative of this equation, we obtain:
\begin{equation*}
(1+\frac{\kappa}{2} s_g)\dd[\dd_gs_g(h)] = 5e^{2\phi}\,\dd\xi   \, .
\end{equation*}
On the other hand, linearizing Equation \eqref{eq:ricdphi} at $(g,\phi)$ gives:
\begin{equation*}
 \dd[\dd_gs_g(h)]=\frac{2}{3}s_g\,\dd\xi \, ,     
\end{equation*}

\noindent
which substituted back into the previous equation yields:
\begin{equation*}
0 = (  (1+\frac{\kappa}{2}s_g) \frac{2}{3}s_g - 5e^{2\phi}) \dd\xi\, ,    
\end{equation*}

\noindent
Substituting now the background values $\kappa s_g=-24$, $\kappa e^{2\phi}=48$ in this equation, we obtain:
\begin{equation*}
- \frac{64}{\kappa} \,\dd\xi=0  \, ,
\end{equation*}

\noindent
and hence $\xi$ is constant. We prove next that $\xi=0$.

\begin{proposition}\label{prop:dilatonrigid}
Let $(g,\phi)$ be a hyperbolic Heterotic soliton with vanishing torsion. For every $(h,\xi)\in\mathbb{E}_{(g,\phi)}$ we have $\xi=0$ and $\dd_g s_g(h)=0$.
\end{proposition}

\begin{proof}
Since $\xi$ is constant, the linearization of the dilaton equation \eqref{eq:dilaton} at $(g,\phi)$ reads:
\begin{equation*}
-2e^{2\phi}\xi + \kappa\,\dd_g(|\cR^g|_g^2)(h)=0.
\end{equation*}
Using Lemma~\ref{lemma:linearcurvature} together with the background values $\kappa s_g=-24$, $\kappa e^{2\phi}=48$, this equation reduces to:
\begin{equation*}
-\frac{96}{\kappa}\,\xi + \kappa\cdot\frac1{6}\cdot\frac{-24}\kappa\dd_g s_g(h)=0 \, ,
\end{equation*}
whence $\kappa\,\dd_g s_g(h) = -24\,\xi$. On the other hand, linearizing the trace of the Einstein equation \eqref{eq:einstein} yields:
\begin{equation*}
\dd_g s_g(h) - 3e^{2\phi}\xi + 2\kappa\cdot\frac16{s_g}\,\dd_g s_g(h)=0.
\end{equation*}
Substituting the same background values and the expression for $\dd_g s_g(h)$ gives
\begin{equation*}
0=-7\left(-\frac{24}{\kappa}\,\xi\right) - \frac{144}{\kappa}\,\xi = \frac{24}{\kappa}\,\xi \,,
\end{equation*}
hence $\xi=0$ and consequently $\dd_g s_g(h)=0$.
\end{proof}

\noindent
Next, we show that essential deformations are automatically transverse‑traceless.

\begin{lemma}\label{lemma:traceessential}
Let $(g,\phi)$ be a three‑dimensional hyperbolic Heterotic soliton. For every $(h,\xi)\in\mathbb{E}_{(g,\phi)}$ we have $\tr_g h = 0$.
\end{lemma}

\begin{proof}
Since $\nabla^{g \ast} h=0$ by the slice condition, the linearization \eqref{eq:linearscalar} of the scalar curvature on an Einstein manifold simplifies to
\begin{equation*}
\dd_g s_g(h) = \delta^g\dd \,(\tr_g h) - \frac13s_g \,\tr_g h,
\end{equation*}
Proposition \ref{prop:dilatonrigid} gives $\dd_g s_g(h)=0$, hence
\begin{equation*}
\delta^g\dd \,(\tr_g h) - \frac13s_g \,\tr_g h = 0 \,.
\end{equation*}
Multiplying by $\tr_g h$ and integrating over $M$ yields
\begin{equation*}
0 = \int_M \left( \tr_g h \,\delta^g\dd \,(\tr_g h) - \frac13s_g\,(\tr_g h)^2 \right) \,\nu_g = \int_M \left( |\dd (\tr_g h)|_g^2 - \frac13s_g\,(\tr_g h)^2 \right) \,\nu_g  \,.
\end{equation*}
Since $s_g < 0$, the two terms in the integrand are non-negative, thus both are necessarily zero pointwise. In particular, $-s_g\,(\tr_g h)^2 = 0$, which implies $\tr_g h = 0$.
\end{proof}

\noindent
We now characterize $h$ as an infinitesimal Einstein deformation.

\begin{lemma}\label{lemma:einsteindef}
Let $(g,\phi)$ be a three‑dimensional hyperbolic Heterotic soliton. For every $(h,\xi)\in\mathbb{E}_{(g,\phi)}$, $h$ is an infinitesimal Einstein deformation.
\end{lemma}

\begin{proof}
From Proposition~\ref{prop:dilatonrigid} and Lemma~\ref{lemma:traceessential} we know $\xi=0$, $\tr_g h=0$ and $\nabla^{g \ast} h=0$. Then, linearizing \eqref{eq:einstein} and using \eqref{eq:linearriemannsquare} yields
\begin{align*}
0 &= \dd_g\Ric^g(h) - \frac12{e^{2\phi}}h + \kappa\, \dd_g(\cR^g \circ_g\cR^g) \\
&= (1 + \frac\kappa3{s_g}) \,\dd_g\Ric^g(h) - (\frac12{e^{2\phi}} + \frac\kappa{18}{s_g^2}) \,h\,.
\end{align*}
Now, taking \eqref{eq:backg} into account, we obtain
\begin{equation*}
0 = -7\dd_g\Ric^g(h) - (\frac{24}\kappa + \frac{24^2}{18\kappa}) \,h = -7\dd_g\Ric^g(h) - \frac{56}\kappa \,h = -7\dd_g\Ric^g(h) + \frac73s_g h \,,
\end{equation*}
whence
\begin{equation}\label{eq:linearizedricci}
\dd_g\Ric^g(h) = \frac13s_g h \,.
\end{equation}

\noindent
After all these substitutions, using \eqref{eq:linearriemannsquare} and \eqref{eq:backg}, the linearized Einstein equation \eqref{eq:einstein} reads after simplification:
\begin{equation*}
\dd_g\Ric^g(h) = \frac13s_g h \,.
\end{equation*}
For a transverse‑traceless tensor on an Einstein manifold, \eqref{eq:linearricci} gives
\begin{equation}\label{eq:linearizedricci2}
\dd_g\Ric^g(h) = \frac{1}{2}\Delta_L h = \frac{1}{2}\left(\nabla^{g \ast}\nabla^gh + \frac23s_gh -2\cR_0(h)\right) \,,
\end{equation}
where we have used $h\circ_g \Ric^g + \Ric^g\circ_g h = \frac23s_gh$.
From \eqref{eq:linearizedricci} and \eqref{eq:linearizedricci2} we thus obtain
\begin{equation}\label{eq:linearizedeinsteinstructure}
\nabla^{g \ast}\nabla^gh - 2\cR_0(h) = 0 \,.
\end{equation}

\noindent
By \cite{BergerEbin} (see also \cite[\S12.30]{Besse}), \eqref{eq:linearizedeinsteinstructure} is equivalent to $h$ being an infinitesimal Einstein deformation.
\end{proof}

We are now ready to prove our first main result.

\begin{theorem}\label{thm:rigidityessential}
Let $(g,\phi)$ be a non-flat three‑dimensional compact Heterotic soliton with vanishing torsion and constant dilaton. Then, the vector space $\mathbb{E}_{(g,\phi)}$ of essential deformations of $(g,\phi)$ is zero-dimensional.
\end{theorem}

\begin{proof}
By \eqref{prop:classificationphiconstant}, the metric $g$ is hyperbolic. Let $(h,\xi)\in \mathbb{E}_{(g,\phi)}$ be an essential deformation of the hyperbolic soliton $(g,\phi)$. By Proposition \ref{prop:dilatonrigid}, we have $\xi = 0$. Furthermore, by Lemma \ref{lemma:einsteindef}, $h$ is an essential infinitesimal Einstein deformation of the hyperbolic metric $g$. Then, by a classical result of Koiso \cite{KoisoII} (see also \cite[\S12.73]{Besse}), it follows that $h=0$, and thus $\mathbb{E}_{(g,\phi)} = 0$.
\end{proof} 


\section{Heterotic solitons with harmonic curvature}
\label{sec:harmonic}


In this section, we prove that if the Riemann curvature tensor of a three‑dimensional compact Heterotic soliton with vanishing torsion is divergence‑free, then the dilaton is necessarily constant. Combined with Proposition \ref{prop:classificationphiconstant}, this shows that the only non‑flat solutions in this class are the hyperbolic ones and adds a layer of difficulty to the construction of compact Heterotic solitons with non-constant dilaton.  

\begin{definition}
A Riemannian manifold $(M,g)$ has \emph{harmonic curvature} if its Riemann tensor is divergence‑free, i.e. $\dd_{\nabla^g}^{\ast} \cR^g = 0$.
\end{definition}

\noindent
In three dimensions, the harmonicity of the Riemann tensor is equivalent to the Ricci tensor being Codazzi \eqref{eq:bianchi}, which, in turn, is equivalent to $(M,g)$ being locally conformally flat with constant scalar curvature. We leave open the existence of locally conformally flat torsionless Heterotic solitons with non-constant dilaton, and therefore non-constant scalar curvature. 

\begin{theorem}\label{thm:harmonic}
Let $M$ be a closed three-dimensional manifold and $(g,\phi)$ a non-flat Heterotic soliton with vanishing torsion. If $g$ has harmonic curvature, then $(g,\phi)$ is hyperbolic.
\end{theorem}

\begin{proof}
Assume that $(g,\phi)$ satisfies the system \eqref{eq:einstein}, \eqref{eq:yangmills}, and \eqref{eq:dilaton} on a compact oriented three‑manifold $M$. By hypothesis $\dd_{\nabla^g}^{\ast} \cR^g = 0$, whence the scalar curvature $s_g$ is constant. Then, from Corollary~\ref{cor:ricdphi} we have:
\begin{equation}
\label{eq:ricdphi0}
\Ric^g(\dd\phi) = 0 \,.
\end{equation}

\noindent
Let $U := \{p\in M\mid \dd\phi_p\neq0\}$ be the open subset of $M$ where the gradient of $\phi$ is non-vanishing. Assume $U$ is non-empty. On $U$ define the unit 1-form:
\begin{equation*}
 u := \frac{1}{\vert\dd\phi\vert_g} \dd\phi \,.  
\end{equation*}
From \eqref{eq:ricdphi0}, $\Ric^g$ has eigenvalue $0$ in the direction of $\dd\phi$ on $U$. \medskip

\noindent
On the other hand, since the harmonicity of $\cR^g$ is equivalent to $\dd^{\nabla^g} \Ric^g = 0$, the Yang–Mills equation \eqref{eq:yangmills} reduces to $\cR^g(\dd\phi)=0$, or, equivalently, taking the interior product with an arbitrary tangent field $v_1\in\mathfrak{X}(M)$ and using \eqref{eq:riemann2},
\begin{equation}
\label{eq:yangmillscontract}
\dd\phi \wedge \left( \frac12{s_g}v_1 - \Ric^g(v_1) \right) = 0 \,.
\end{equation}
Therefore, $\Ric^g$ has eigenvalue $\frac{s_g}2$ in the directions orthogonal to $\dd\phi$ on $U$. \medskip

\noindent
Moreover, \eqref{eq:ricdphi0} and \eqref{eq:yangmillscontract} together show that we can write the Ricci tensor on $U$ as
\begin{equation}
\label{eq:ricciform}
\Ric^g = \frac12{s_g}\left(g - u \otimes u\right) \,. 
\end{equation}
From \eqref{eq:ricciform} we calculate $\Ric^g\circ_g\Ric^g = \frac{s_g^2}{4}(g-u \otimes u)$ and
\begin{equation}
\label{eq:riccinormU}
|\Ric^g|_g^2 = \frac12{s_g^2}\,.
\end{equation}

\noindent
Substituting \eqref{eq:normriemann} and \eqref{eq:riccinormU} in the scalar identity \eqref{eq:scalaridentity}, we obtain at each point of $U$:
\begin{equation*}
|\dd\phi|_g^2 - \frac{5}{2}e^{2\phi} + s_g + \frac{3\kappa}{4}s_g^2 = 0 \,.
\end{equation*}
Define the function $f:=|\dd\phi|_g^2 - \frac{5}{2}e^{2\phi}$ on $M$. The equality above shows that
\begin{equation*}
f = - s_g - \frac{3\kappa}{4} s_g^2
\end{equation*}
at each point of $U$, and, since $s_g$ is constant on $M$, $f$ is constant on $U$. However, $f$ is, by construction, also locally constant on the interior of the complement of $U$ in $M$. But the union of $U$ with the interior of its complement is dense on $M$, thus it follows that the above function is locally constant on $M$, hence constant since $M$ is connected. Therefore, $\phi$ takes the same value at a maximum and minimum, i.e. $\phi$ is constant. Then, by Proposition~\ref{prop:classificationphiconstant} the soliton is hyperbolic.
\end{proof}

\noindent
Combining this result with Corollary \ref{cor:rigidity} and Theorem \ref{thm:rigidityessential} we obtain Theorem \ref{thm:rigidityessentialharmonic}, as presented in the introduction.


\phantomsection
\bibliographystyle{JHEP} 


\end{document}